\documentclass[a4paper]{article}
\usepackage{amsmath, amsfonts}
\usepackage{amssymb, latexsym}
\usepackage{amsthm}
\usepackage{mathtools}
\usepackage{comment}

\usepackage{tikz}
\usetikzlibrary{arrows.meta}
 \usetikzlibrary{patterns}

\usepackage{easy-todo}

\usepackage[shortlabels]{enumitem}
\setlist[enumerate]{leftmargin=*,align=left,labelindent=\parindent}

\usepackage[%dvipdfmx,
colorlinks=true,       % false: boxed links; true: colored links
linkcolor=blue,          % color of internal links (change box color with linkbordercolor)
citecolor=blue,        % color of links to bibliography
plainpages=false,      % do page number anchors as plain Arabic
pdfpagelabels,
]{hyperref}

\makeatletter
\def\widebreve{\mathpalette\wide@breve}
\def\wide@breve#1#2{\sbox\z@{$#1#2$}%
     \mathop{\vbox{\m@th\ialign{##\crcr
\kern0.08em\brevefill#1{0.8\wd\z@}\crcr\noalign{\nointerlineskip}%
                    $\hss#1#2\hss$\crcr}}}\limits}
\def\brevefill#1#2{$\m@th\sbox\tw@{$#1($}%
  \hss\resizebox{#2}{\wd\tw@}{\rotatebox[origin=c]{90}{\upshape(}}\hss$}
\makeatletter

\DeclareMathOperator{\imp}{\,\rightarrow\,}

\DeclareMathOperator{\defeql}{\stackrel{\textup{def}}{\  =\  }}
\newcommand{\dotminus}{\mathbin{\ooalign{\hss\raise.6ex\hbox{$\cdot$}\hss\crcr$-$}}}

\newcommand{\ELZero}{\mathrm{EL_0}}
\newcommand{\EL}{\mathrm{EL}}

\newcommand{\RCA}{\mathrm{RCA_0}}
\newcommand{\T}{\mathrm{T}}

\newcommand{\BISH}{\mathrm{BISH}}
\newcommand{\qfAC}{\mathrm{QF\text{-}AC_{00}}}

\newcommand{\WKL}{\mathrm{WKL}}

\newcommand{\LLPO}{\mathrm{LLPO}}

\newcommand{\IVT}{\mathrm{IVT}}

\newcommand{\MIN}{\mathrm{MIN}}

\newcommand{\UInt}{[0,1]}
\newcommand{\Real}{\mathbb{R}}
\newcommand{\Rat}{\mathbb{Q}}

\newcommand{\Nat}{\mathbb{N}}

\newcommand{\Baire}{\Nat^\Nat}
\newcommand{\BE}{\mathrm{BE}}
\newcommand{\Bin}{\Two^\ast}
\newcommand{\BTree}{\Two^\Nat}
\newcommand{\Seq}{\Nat^\ast}
\newcommand{\Two}{\left\{0,1 \right\}}

\newcommand{\nil}{\langle \,\rangle}

\newcommand{\lth}[1]{\lvert #1 \rvert}
\newcommand{\Sd}{\mathsf{Sd}}
\newcommand{\mesh}{\mathsf{mesh}}

\newcommand{\Dist}[2]{d(#1,#2)}

\newtheorem{theorem}{Theorem}[section]
\newtheorem{lemma}[theorem]{Lemma}

\theoremstyle{definition}

\theoremstyle{remark}
\newtheorem{remark}[theorem]{Remark}
\newtheorem{notation}[theorem]{Notation}

\numberwithin{equation}{section}

\title{Constructive equivalence between Brouwer's fixed-point
theorem and weak K\"onig's lemma}
\author{Tatsuji Kawai\footnote{
  Email: \texttt{tatsuji.kawai@kochi-u.ac.jp}}\\
  \small
  Department of Information Science, Kochi University\\
  \small
  2-5-1 Akebono-cho, Kochi 780-8520 Japan
}
\date{}
\begin{document}
\maketitle
\begin{abstract}
In the context of constructive reverse mathematics, we show that
Brouwer's fixed-point theorem and weak K\"{o}nig's lemma ($\WKL$) are
equivalent. To derive $\WKL$ from Brouwer's fixed-point theorem, the
construction of a continuous function on the unit square without
fixed points due to Orevkov [Soviet Math.\ Doklady (1963), 1253--1256]
is generalised to yield a uniformly continuous function on the unit
square whose fixed points encode information about infinite paths
of a given infinite tree.

\medskip

 \\[3pt]
\end{abstract}

%\listoftodos

\section{Introduction}
\label{sec:Introduction}
Over the base system $\RCA$ of classical reverse
mathematics~\cite{SimpsonBook}, weak K\"{o}nig's lemma ($\WKL$) is
equivalent to Brouwer's fixed-point theorem~\cite{ShiojiTanaka}
\cite[IV.7]{SimpsonBook}. On the other hand, in Bishop's constructive
mathematics ($\BISH$) \cite{Bishop-67,Bishop-Bridges-85}, an informal
mathematics that is based on intuitionistic logic and assumes certain
function existence axioms including the axiom of countable choice,
Brouwer's fixed-point theorem is equivalent to 
the \emph{lesser limited principle of omniscience}
($\LLPO$)~\cite{Hendtlass2012FixedPT}, 
which is in turn equivalent to $\WKL$ \cite{IshiharaLLPOWKLHahn}.  Thus,
Brouwer's fixed-point theorem is equivalent to $\WKL$ over $\BISH$.
In fact, many mathematical theorems that 
are either equivalent to $\WKL$ 
or provable in $\RCA$ become equivalent to $\LLPO$ over $\BISH$.
Examples include Brouwer's
fixed-point theorem and Cantor's intersection theorem
\cite{IshiharaLLPOWKLHahn}, which are equivalent to $\WKL$,
and the binary expansion of real numbers in $\UInt$ and the
intermediate value theorem \cite{IshiharaEtalBinaryExpansion},
which are provable in $\RCA$. 
Hence, the equivalence between Brouwer's fixed-point
theorem and $\WKL$ over $\BISH$ cannot be directly compared to the
equivalence between these two statements in classical reverse
mathematics.

The aim of this paper, therefore, is to establish the equivalence between
Brouwer's fixed-point theorem and $\WKL$ in the context of
\emph{constructive reverse mathematics}~\cite{ConstRevMatheCompactness}.
By constructive reverse mathematics, we mean mathematics based on
a formal system $\T$ that is contained in both $\BISH$ and $\RCA$. 
In this context, several equivalences related to
Brouwer's fixed-point theorem are known:
\begin{enumerate}
  \item the Weak Fan Theorem and Brouwer's Approximate
    Fixed-Point Theorem \cite[Theorem 7]{Veldman2009}; 
  \item Kleene's alternative to the Weak Fan Theorem and
      the existence of a continuous function on the unit square
      without fixed points \cite[Theorem 10]{Veldman2009};
  \item the intermediate value theorem and the convex version of
    $\WKL$ \cite[Theorem 3]{IshiharaEtalBinaryExpansion}.
\end{enumerate}
In the first equivalence, the Weak Fan Theorem
is the contraposition of $\WKL$~\cite{IshiharaWKLimpliesFan},
whereas Brouwer's Approximate Fixed-Point Theorem is the statement
that for every $\varepsilon > 0$ and for every (pointwise) continuous function $f$
on the unit square, there exists a point $x$ whose distance to $f(x)$
is less than $\varepsilon$.
The second equivalence is the contrapositive form of the equivalence
between $\WKL$ and Brouwer's fixed-point theorem. Here, the Kleene's
Alternative to the Weak Fan Theorem is the statement that there exists
an infinite binary tree such that every binary sequence eventually
leaves the tree
\cite[Lemma 9.8]{KleeneVesley}\cite{VeldmanFANKleene}\cite[4.7.6]{ConstMathI}.
This equivalence provides a constructive alternative to the classical
proof. In fact, the derivation of $\WKL$ from Brouwer's fixed-point
theorem in \cite[IV.7.7]{SimpsonBook} is identical to the constructive
proof of one direction of the second equivalence
(cf.\ Remark~\ref{rm:Replacement}).
The last equivalence can be seen as a one-dimensional version of the
required equivalence between $\WKL$ and Brouwer's fixed-point theorem. 

In this paper, after reviewing the base system $\ELZero$ for
constructive reverse mathematics~\cite{IshiharaEtalBinaryExpansion},
we first show that $\WKL$ implies Brouwer's fixed-point theorem. The
proof consists of a constructive proof of Brouwer's approximate
fixed-point theorem for uniformly continuous functions via Sperner's
lemma, followed by an application of the \emph{minimum principle}
($\MIN$) to obtain an exact fixed-point. This part is a
straightforward adaptation of the similar proof in 
$\BISH$ \cite{Hendtlass2012FixedPT} but our proof remains
strictly within $\ELZero$.
See Section~\ref{sec:WKLImpBrouwer}.
 
Next, we show that Brouwer's fixed-point theorem implies $\WKL$.  The
proof is inspired by the construction of a continuous function on the
unit square without fixed points in the classical proof \cite[IV.7.7]
{SimpsonBook}, which is originally due to Orevkov \cite{Orevkov1963}.
Here, this construction is generalised to yield a uniformly continuous
function on the unit square whose fixed points encode information about
infinite paths of a given infinite tree.
See Section~\ref{sec:BrouwerImpWKL}.

\section{Formal base system}
We adopt $\ELZero$ \cite[Section 2]{IshiharaEtalBinaryExpansion}
as our base system for constructive reverse mathematics.
$\ELZero$ is a subsystem of elementary analysis $\EL$ \cite[3.6]{ConstMathI},
which is based on two-sorted intuitionistic logic (one
sort for natural numbers and the other for functions on natural
numbers). $\ELZero$ is obtained from $\EL$ by restricting the induction
scheme to quantifier-free formulas. The classical system $\RCA$
is obtained from $\ELZero$ by adding the law of excluded middle.  For
a detailed description of $\ELZero$, the reader is referred to
\cite[Section 2]{IshiharaEtalBinaryExpansion}. 
For the purpose of our investigation, however, some comments are in
order.

First, $\ELZero$ has only the \emph{quantifier-free axiom of choice}:
\[
  \qfAC: \forall m \exists n A(m,n) \imp \exists \alpha \forall m
  A(m,\alpha(m)),
\]
where $m,n$ range over natural numbers and $\alpha$ ranges over
functions on natural numbers, and $A$ is a quantifier-free formula that
does not contain $\alpha$ as a free variable.

Second, since $\ELZero$ has only two sorts, a uniformly continuous
function on real numbers (or more generally Euclidean spaces) must be
defined in terms of a function on natural numbers
(see \cite[Section 5]{Ishihara2009Relativize} for the details of the encoding).
However, as in \cite{IshiharaEtalBinaryExpansion}, we work with
uniformly continuous functions on Euclidean spaces as if such higher
type objects are directly available in $\ELZero$.

\paragraph{Notations}
The letters $i,j,k,l,m,n,\dots$ range over the set $\Nat$ of natural numbers,
and $\alpha,\beta,\dots$ range over the set $\Nat^{\Nat}$ of functions
on natural numbers.
Using a pairing function, one can code finite sequences of natural
numbers in $\ELZero$. The set of finite sequences of natural numbers
is denoted by $\Seq$, the set of binary sequences is denoted by
$\Bin$, and letters $a,b,c,\dots$ range over $\Seq$. Also, $\nil$
denotes the empty sequence, and $\langle i_0, \dots, i_{n-1} \rangle$
denotes a finite sequence of length $n$. The concatenation of finite
sequences $a$ and $b$ is denoted by $a*b$, the length of $a$ is
denoted by $|a|$, for $i < |a|$, the $i$-th component of $a$ is
denoted by $a_i$, and $a \preceq b$ means that $a$ is an initial
segment of $b$. Lastly, for an infinite sequence $\alpha$,
$\overline{\alpha}n$ denotes the initial segment of $\alpha$ of
length $n$, namely $\langle \alpha(0),\dots,\alpha(n-1)\rangle$.
We also define $\overline a n = \langle a_{0},\dots,a_{n-1} \rangle$
for a finite sequence $a$ whenever $n \leq \lth{a}$.

\section{\texorpdfstring{$\WKL$}{WKL} implies Brouwer's fixed-point theorem}
\label{sec:WKLImpBrouwer}
We first recall some notions related to $\WKL$.
A \emph{tree} is a detachable subset%
\footnote{We identify a detachable subset with its characteristic
function.} $T \subseteq \Seq$ such that 
$\nil \in T$  and $a \in T \land b \preceq a \imp b \in T$
for all $a, b \in \Seq$. Let $T$ be a tree:
\begin{itemize}
  \item $T$ is \emph{binary} if $T \subseteq \Bin$.

  \item $T$ is \emph{infinite} if $\forall n \in \Nat \exists a
    \in \Seq \left( |a| = n \land a \in T \right)$.

  \item An \emph{infinite path} of $T$ is a sequence 
    $\alpha \in \Baire$ such that
    $\forall n \in \Nat \; \overline{\alpha}n \in T$.
\end{itemize}
\emph{Weak K\"{o}nig's lemma} ($\WKL$) is the following principle:
\begin{description}
  \item[$\WKL$:]
    Every infinite binary tree has an infinite path.
\end{description}

% It can be shown that $\WKL$ is equivalent to the following form
% of K\"{o}nig lemma (called \emph{Bounded K\"{o}nig's lemma})
% \cite[Proposition 4.7]{FujiwaraDFT}:
% \begin{description}
%   \item[$\WKLf$:] Every infinite bounded tree has an infinite path.
% \end{description}
% Here, a tree $T \subseteq \Seq$ is \emph{bounded}
% if it has a \emph{height-wise bounding function}, i.e.,
% a function $f \colon \Nat \to \Nat$ such that
% \[
%   \forall a \in T \forall i < |a| \left(a_i \leq f(i)  \right).
% \]

In order to state Brouwer's fixed-point theorem, we recall the notion
of the standard $n$-simplex. Fix a natural number $n \in \Nat$.
The standard $n$-simplex, denoted $\Delta_{n}$, is the convex hull of
$e_1, \dots, e_{n+1}$, where $e_{i} $ is the $i$-th standard unit vector in
$\Real^{n+1}$:
\[
  \Delta_{n}
  \defeql
  \left\{ \sum_{j=1}^{n+1}a_{j}e_{j}
    \mid
    \sum_{j=1}^{n+1}a_{j} = 1, a_{j} \geq 0 \right\}.
\]
Then Brouwer's fixed-point theorem states:
\begin{quote}
  For any $n \in \Nat$, any uniformly continuous function
  $f \colon \Delta_{n} \to \Delta_{n}$ has a fixed-point.
\end{quote}

As a preliminary step for deriving Brouwer's fixed-point theorem from
$\WKL$, we first prove Brouwer's approximate fixed-point theorem
in $\ELZero$. The proof follows the well-known algorithm for finding
an approximate fixed-point via Sperner's lemma~\cite{Todd1976}.
\footnote{In \cite{Veldman2009}, Brouwer's approximate fixed-point
theorem is formulated for pointwise continuous functions, which was
shown to be equivalent to Brouwer's fan theorem. On the other hand,
the version of this paper is formulated for uniformly continuous
functions.}

Let $(P,\leq)$ be a finite poset where $\leq$ is decidable.
A \emph{chain} in $P$ is an inhabited finite sequence $\langle a_{0},
a_{1}, \dots,a_{l} \rangle$ in $P$ such that
\begin{equation*}
  a_{0} < a_{1} < \cdots < a_{l}.
\end{equation*}
Let ${\Sd(P)}$ denote the poset of all chains in $P$
ordered by inclusion.
A chain $a_{0} < a_{1} < \cdots < a_{l}$ of length $l+1$ in $P$ is
called an $l$-simplex of $P$. As an element of $\Sd(P)$, such a
chain is denoted by $\langle a_{0}, \dots, a_{l} \rangle$.%
\footnote{The notation of a simplex (or chain) conflicts with that
of finite sequence; however, it should be clear from the context
which of these is intended.}

For a fixed $n \in \Nat$, let $[n]$ denote the poset $\left\{ 0,\dots,
n \right\}$  with the usual order $0 < 1 < \dots < n$.
Then the poset $\Sd([n])$ is called the \emph{barycentric subdivision} of $[n]$.%
\footnote{Specifically, the simplices (or
chains) of $\Sd([n])$ form the barycentric subdivision of $[n]$.}
More generally, for each $k \in \Nat$, define $\Delta^{k}$ (called the
$k$-th barycentric subdivision of $[n]$) inductively by
\begin{align*}
  \Delta^{0} &\defeql [n],  &
  \Delta^{k+1} &\defeql \Sd(\Delta^{k}).
\end{align*}
Set 
$
\Delta \defeql \bigcup_{k \in \Nat} \Delta^{k}.
$
\begin{notation}
  Elements of $\Delta$ are denoted by small Greek letters
  $\sigma,\tau,\mu,\dots$. For each $k \in \Nat$,
  $\sigma \in \Delta^{k}$, and
  $\tau \in \Delta^{k+1}$, we write $\sigma \in \tau$ if
  $\sigma$ is an element of the chain $\tau$.%
\end{notation}
For each $k \in \Nat$ and $\sigma \in \Delta^{k}$, we assign
the \emph{barycentric coordinate of $\sigma$}, denoted $\sigma^{*}$,
by induction on $k$ as follows:
\begin{align*}
  i^{*} &\defeql e_{i+1} && (i \in [n]),\\
  \langle \sigma_{0}, \dots, \sigma_{m} \rangle^{*}
  &\defeql
  \frac{1}{m+1} \sum_{j = 0}^{m}\sigma_{j}^{*}
  && (\langle \sigma_{0}, \dots, \sigma_{m} \rangle \in \Delta^{k+1}).
\end{align*}
For $\sigma \in \Delta$ and $i \in [n]$, we say that $\sigma$ is on
the \emph{$(i+1)$-th facet} of $\Delta_n$ if no membership chain $i
\in  \cdots \in \sigma$ exists in $\Delta$. It is easy to see that
$\sigma$ is on the $(i+1)$-th facet if and only if $(\sigma^{*})_{i+1} = 0$. 

For each $k \in \Nat$, a \emph{labelling} of $\Delta^{k}$ is a
function $L \colon \Delta^{k} \to [n]$. A labelling is a \emph{Sperner
labelling} if 
\[
  (\sigma^{*})_{L(\sigma) + 1} > 0
\]
for all $\sigma \in \Delta^{k}$.%
\footnote{
  The condition is equivalently expressed as follows: 
  if $\sigma$ is on the $(i+1)$-th facet of $\Delta_n$, then $L(\sigma) \neq i$.
}
An $n$-simplex $\langle \sigma_{0}, \dots, \sigma_{n} \rangle$ in
$\Delta^{k}$ is \emph{fully labelled} if 
\[
  \left\{ L(\sigma_i) \mid 0 \leq i \leq n \right\} = [n].
\]
The following is a special case of Sperner's lemma \cite{SpernerLemma}.
\begin{lemma}[Sperner's Lemma]
  \label{lem:Sperner}
    Let $k \in \Nat$, and let $L \colon \Delta^{k} \to [n]$ be a
    Sperner labelling of $\Delta^{k}$.  Then there exists an odd
    number of fully labelled $n$-simplices.
\end{lemma}
\begin{proof}
  The standard proof of Sperner's lemma, for example 
  \cite[Lemma 1.3]{CourseInTopologicalCombinatorics}, 
  is already constructive and can be carried out in $\ELZero$.
\end{proof}

We now prove the approximate version of Brouwer's
fixed-point theorem.
\begin{theorem}[Brouwer's approximate fixed-point theorem]
  \label{thm:ApproximateBrouwer}
  Let $n \in \Nat$, and let 
  $f \colon \Delta_{n} \to \Delta_{n}$ be a uniformly continuous
  function. Then, for any $k \in \Nat$, 
  there exists $x \in \Delta_{n}$ such that $d(x,f(x)) < 2^{-k}$.
\end{theorem}
\begin{proof}
  Let $\omega$ be a modulus of uniform continuity of $f$.
  We may assume that $\omega(k) \geq k$ for each $k \in \Nat$.
  Without loss of generality, we use the max norm of $\Real^{n+1}$.
  Thus, for $x, y \in \Real^{n+1}$,
  \[
    \Dist{x}{y} \defeql \max\left\{ |x_i - y_i| \mid i \leq n+1
    \right\}.
  \]

  For each $k \in \Nat$, let $k' \in \Nat$ be the least number such that
  \begin{equation}
    \label{eq:kPrime}
    2n \cdot 2^{-k'} \leq 2^{-k}.
  \end{equation}
  Let $l(k)$ be the least number such that 
  \begin{equation}
    \label{eq:lk}
    \mesh(\Delta^{l(k)}) \leq 2^{-\omega(k'+1)},
  \end{equation}
  where
  $
  \mesh(\Delta^{i})
  \defeql \max\left\{ \Dist{\sigma^{*}}{\tau^{*}} \mid \sigma,\tau
  \in \Delta^{i} \right\}
  $
  for each $i \in \Nat$.
  Such $l(k)$ exists since 
  $\mesh(\Delta^{i})$ is bounded by $\big(\frac{n}{n+1}\big)^{i}$
  (see e.g., \cite[Appendix I]{VickHomology}).
  Then, for any pair of neighbouring elements $\sigma,\tau$
  of $\Delta^{l(k)}$,%
  \footnote{This means that $\sigma$ and $\tau$ are comparable as
  elements of the ordered set $\Delta^{l(k)}$.}
  \begin{align}
    \label{eq:neighbour}
    \begin{aligned}
      &\Dist{\sigma^{*}}{\tau^{*}} \leq \mesh(\Delta^{l(k)})
            \leq 2^{-\omega(k'+1)} \leq 2^{-(k'+1)},\\
            &\Dist{f(\sigma^{*})}{f(\tau^{*})} \leq  2^{-(k'+1)},
    \end{aligned}
  \end{align}
  where the latter follows from the uniform continuity of $f$.

  By $\qfAC$, there exists a function 
  $\lambda \colon  \sum_{k \in \Nat} \Delta^{l(k)} \to \Two$ such that
  \begin{align*}
    \lambda(k,\sigma) = 0 &\imp
    \Dist{\sigma^{*}}{f(\sigma^{*})} > n \cdot 2^{-k'}, \\
    \lambda(k,\sigma) = 1 &\imp
    \Dist{\sigma^{*}}{f(\sigma^{*})} < 2n \cdot 2^{-k'}
  \end{align*}
  for $k \in \Nat$ and $\sigma \in \Delta^{l(k)}$.

  Fix $k \in \Nat$, and suppose $\lambda(k,\sigma) = 0$
  for all $\sigma \in \Delta^{l(k)}$. For each $\sigma \in
  \Delta^{l(k)}$, since $\sigma,f(\sigma) \in \Delta_{n}$, we have
  \begin{equation}
    \label{eq:OnSimplex}
    \sum_{i = 0}^{n}\left(  (\sigma^{*})_{i+1} - f(\sigma^{*})_{i+1} \right) 
    = 
    \sum_{i = 0}^{n} (\sigma^{*})_{i+1} - \sum_{i = 0}^{n}
    f(\sigma^{*})_{i+1}
    = 1 - 1 = 0.
  \end{equation}
  Since 
  $\Dist{\sigma^{*}}{f(\sigma^{*})} > n \cdot 2^{-k'}$,
  there exists $i \leq n$ such that 
  \[
  (\sigma^{*})_{i+1} - f(\sigma^{*})_{i+1} > 2^{-k'}.
  \]
  Hence, by the finite $\qfAC$, we obtain a labelling $L \colon
  \Delta^{l(k)} \to [n]$ of elements of $\Delta^{l(k)}$ such that
  \[
    (\sigma^{*})_{L(\sigma)+1} - f(\sigma^{*})_{L(\sigma)+1} > 2^{-k'}
  \]
  for each $\sigma \in \Delta^{l(k)}$.
  Note that if $\sigma$ is on the $(i+1)$-th facet of $\Delta_{n}$
  for some $i \in [n]$, then $(\sigma^{*})_{i+1} = 0$, and so
  \[
    L(\sigma) \neq i.
  \]
  Thus, $L$ is a Sperner labelling of $\Delta^{l(k)}$. By 
  Sperner's lemma (Lemma \ref{lem:Sperner}), there exists a fully labelled
  $n$-simplex $\mu_{0}, \dots, \mu_{n}$ in
  $\Delta^{l(k)}$.%
  \footnote{This means that they form a chain $\mu_{0} < \cdots <
  \mu_{n}$ in $\Delta^{l(k)}$.}
  Without loss of generality, we may assume $L(\mu_{i}) = i$ so that
  \begin{equation}
    \label{eq:Positive}
    (\mu_{i}^{*})_{i+1} - f(\mu_{i}^{*})_{i+1} > 2^{-k'}
  \end{equation}
  for each $i \in [n]$. Fix $i \in [n]$. By \eqref{eq:OnSimplex}
  and \eqref{eq:Positive}, there exists $j \in [n]$ such that 
  \[
    (\mu_{i}^{*})_{j+1} - f(\mu_{i}^{*})_{j+1} < 0.
  \]
  Since $\mu_{i}$
  and $\mu_{j}$ are neighbouring in $\Delta^{l(k)}$, they satisfy \eqref{eq:neighbour}.
  Thus
  \begin{align*}
    (\mu_{j}^{*})_{j+1} - f(\mu_j^{*})_{j+1}
    &= 
    \left[ (\mu_{j}^{*})_{j+1} - (\mu_{i}^{*})_{j+1} \right]
    +
    \left[ (\mu_{i}^{*})_{j+1} -  f(\mu_i^{*})_{j+1} \right]
    + \left[ f(\mu_i^{*})_{j+1} -  f(\mu_j^{*})_{j+1}\right]
    \\
    &< 2^{-(k'+1)} + 2^{-(k'+1)} = 2^{-k'},
  \end{align*}
  contradicting $L(\mu_{j}) = j$ (see \eqref{eq:Positive}). 

  Hence, for each $k \in \Nat$, there exists $\sigma \in
  \Delta^{l(k)}$ such that $\lambda(k,\sigma) = 1$. 
  By \eqref{eq:kPrime}, such $\sigma$ (or, more precisely, $\sigma^{*}$)
  witnesses an approximate fixed-point of $f$ with 
  respect to $2^{-k}$.
\end{proof}

Having obtained Brouwer's approximate fixed-point theorem, 
Brouwer's (exact) fixed-point theorem immediately follows from $\WKL$ via the
equivalent of $\WKL$ called the \emph{minimum principle}
($\MIN$):
\begin{description}
  \item[$\MIN$:]
    Every real-valued uniformly continuous function
    $f \colon X \to \Real$ on a compact metric space $X$ attains its
    minimum: i.e., there exists $z \in X$ such that $f(z) =
    \inf f$.
\end{description}
For the equivalence between $\WKL$ and $\MIN$ in $\ELZero$,
the reader is referred to~\cite{Kawai2026WKLMIN}.

\begin{theorem}
  \label{thm:WKLImpBrouwer}
  $\WKL$ implies Brouwer's fixed-point theorem.
\end{theorem}
\begin{proof}
  Let $n \in \Nat$, and let $f \colon \Delta_{n} \to \Delta_{n}$ be a
  uniformly continuous function. 
  Then the mapping $x \mapsto d(x,f(x))$ is a real-valued
  uniformly continuous function on $\Delta_{n}$, which is compact.
  By Brouwer's approximate fixed-point theorem
  (Theorem \ref{thm:ApproximateBrouwer}),
  we have $\inf_{x \in \Delta_{n}} d(x,f(x)) = 0$. 
  Hence,  by $\MIN$, there exists $x \in \Delta_{n}$
  such that $d(x, f(x)) = 0$, i.e., $x = f(x)$.
\end{proof}

\begin{remark}
  For the proof of Brouwer's approximate fixed-point theorem and
  Brouwer's fixed-point theorem in $\BISH$ (with the latter assuming
  $\LLPO$), see \cite[Section 2.2]{Hendtlass2012FixedPT}.  Note that
  in $\ELZero$, Brouwer's fixed-point theorem does not follow from
  $\LLPO$ as this would imply that $\LLPO$ implies $\WKL$ in $\RCA$,
  which is not the case.
\end{remark}

\section{Brouwer's fixed-point theorem implies \texorpdfstring{$\WKL$}{WKL}}
\label{sec:BrouwerImpWKL}
For the purpose of deriving $\WKL$ from Brouwer's fixed-point
theorem, it suffices to use the two-dimensional instance of the theorem:
\begin{quote}
  Any uniformly continuous function on $\UInt \times \UInt$ has a fixed-point.
\end{quote}
\begin{remark}
  \label{rm:IVTBE}
  The one-dimensional instance of Brouwer's fixed-point theorem
  implies the \emph{intermediate value theorem} ($\IVT$), which states
  \begin{quote}
    If $f \colon \UInt \to \Real$ is a uniformly continuous
    function with $f(0) < 0 < f(1)$, then there exists $x \in \UInt$
    such that $f(x) = 0$.
  \end{quote}
  Berger et al.\ \cite{IshiharaEtalBinaryExpansion} showed that $\IVT$
  implies the \emph{lesser limited principle of omniscience} ($\LLPO$)
  and the \emph{binary expansion of real numbers in the unit interval}
  ($\BE$).
  Here, $\LLPO$ is the principle
  \[
    \forall \alpha,\beta \in \BTree
    \Bigl( 
      \neg \bigl( \exists n \left( \alpha(n) \neq 0 \right) 
                  \wedge
                  \exists n \left( \beta(n) \neq 0 \right)
           \bigr)
      \imp
      \neg \exists n  \left( \alpha(n) \neq 0 \right) 
                  \lor
      \neg \exists n  \left( \beta(n) \neq 0 \right) 
    \Bigr),
  \]
  and $\BE$ is the statement
  \begin{quote}
    For every real number $x \in \UInt$, there exists
    $\alpha \in \BTree$ such that
    $x = \sum_{i = 0}^{\infty} 2^{-(i+1)}\alpha(i)$.
  \end{quote}
  Hence, the two-dimensional instance of Brouwer's fixed-point theorem
  used in the proof of the following theorem suffices to derive both $\LLPO$ and $\BE$.
\end{remark}

\begin{theorem}
  \label{thm:BrouwerImpWKL}
  Brouwer's fixed-point theorem implies $\WKL$.
\end{theorem}
\begin{proof}
  Assume Brouwer's fixed-point theorem, and let $T$
  be an infinite binary tree.
  In what follows, we write $U$ for $\UInt \times \UInt$.

  For each $a \in \Bin$, we associate a dyadic rational interval
  $[l_a,r_a]$ by
  \[
    l_{a} \defeql \sum_{i = 0}^{|a|-1} 2^{-(i+1)} a_i, 
    \qquad
    r_a \defeql l_a+2^{-|a|}.
  \]
  For each $n \in \Nat$, let
  \[
    C_n
    \defeql
    \left\{ a \in \Bin \mid
      \lth{a} \leq n \land
      a \notin T \land 
      \forall i < \lth{a} \left(\overline{a}i \in T \right)
    \right\}.
  \]
  That is, $C_n$ is the set of finite
  binary sequences of length less than $n$
  that have just left the tree $T$.
  Then, define $\overline C_n \subseteq U$ by
  \[
    \overline C_n \defeql \bigcup \left\{ [l_a,r_a] \mid a\in C_n  \right\}.
  \]
  Note that 
  \[ 
    m \le n \rightarrow \overline{C_m} \subseteq \overline{C_n}
  \]
  for all $m,n \in \Nat$. Moreover, the fact that $T$ is infinite is
  equivalent to
  \[
    \forall n \in \Nat \exists a \in \Bin
    \left( |a| = n \land
    \forall i \leq \lth{a} \left(\overline{a}i \notin C_n \right)
    \right).
  \]
  Then, as described in the proof of \cite[Theorem 7]{Veldman2009},
  one can construct by induction on $n\in \Nat$ a sequence of
  uniformly continuous functions $f_n \colon \overline C_n \times
  \overline C_n \to U \; (n \in \Nat)$ such that
  \begin{enumerate}
    \item $f_{n+1}$ is an extension of $f_n$,

    \item 
      $
        \operatorname{Im}(f_n)
        % =
        % \left\{  f_n(z) \mid z \in \overline C_n \times \overline C_n  \right\}
        \subseteq
        \partial U,
      $
      where $\partial U$ is the boundary of $U$, and

    \item $f_n$ fixes $\partial U$, i.e.,
      $
      \forall z \in \partial U  \cap \left( \overline C_n \times \overline
      C_n  \right)  \left(f_n(z)=z  \right).
      $
  \end{enumerate}
  For a detailed construction of such a sequence of uniformly
  continuous functions, see \cite[Lemma 4 and Theorem 7]{Veldman2009}
  or \cite[IV.7.7]{SimpsonBook}. The rest of the proof depends only on
  the above properties 1--3 of $f_n$, and not on any particular
  construction of the functions $f_n$.

  We will construct a uniformly continuous function on $U$ whose
  fixed points lie outside $\overline C_n \times \overline C_n$ for
  every $n \in \Nat$. By the binary expansion, a fixed-point of such a
  function should give rise to a binary sequence whose
  initial segments avoid every $C_n$. 
  Such a sequence determines an infinite path in $T$.

  To this end, we first extend each $f_n \;(n \in \Nat)$ to a
  uniformly continuous function $\overline f_n \colon U \to U$ that
  fixes $\partial U$ as follows:
  \begin{enumerate}

    \item Subdivide $U$ into $2^{n} \times 2^{n}$
      squares so that $U$ is divided into a chess board of size $2^n$.

    \item Triangulate each square in the $2^{n}\times 2^{n}$ grid into
      four triangles by connecting the center of the square with the
      four vertices of the square.  This gives a triangulation of $U$.
      See Figure \ref{fig:ChessBoard}.
      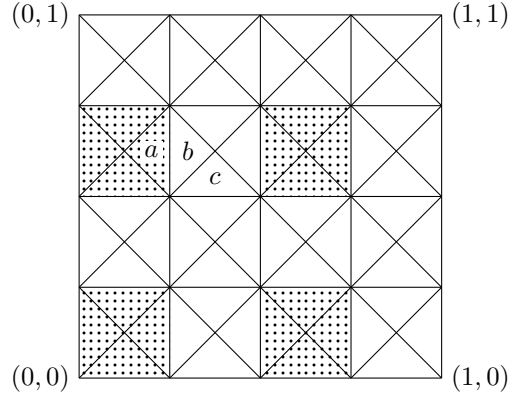
\begin{figure}[tb]
  \centering
\begin{tikzpicture}[scale=1.2]
  %
  % Chess board
  %
  \draw (0,0) -- (4,0);
  \draw (0,1) -- (4,1);
  \draw (0,2) -- (4,2);
  \draw (0,3) -- (4,3);
  \draw (0,4) -- (4,4);
  \draw (0,0) -- (0,4);
  \draw (1,0) -- (1,4);
  \draw (2,0) -- (2,4);
  \draw (3,0) -- (3,4);
  \draw (4,0) -- (4,4);
  %
  % Triangulation
  %
  \draw (0,0) -- (1,1);
  \draw (1,0) -- (2,1);
  \draw (2,0) -- (3,1);
  \draw (3,0) -- (4,1);
  \draw (0,1) -- (1,0);
  \draw (1,1) -- (2,0);
  \draw (2,1) -- (3,0);
  \draw (3,1) -- (4,0);
  \draw (0,1) -- (1,2);
  \draw (1,1) -- (2,2);
  \draw (2,1) -- (3,2);
  \draw (3,1) -- (4,2);
  \draw (0,2) -- (1,1);
  \draw (1,2) -- (2,1);
  \draw (2,2) -- (3,1);
  \draw (3,2) -- (4,1);
  \draw (0,2) -- (1,3);
  \draw (1,2) -- (2,3);
  \draw (2,2) -- (3,3);
  \draw (3,2) -- (4,3);
  \draw (0,3) -- (1,2);
  \draw (1,3) -- (2,2);
  \draw (2,3) -- (3,2);
  \draw (3,3) -- (4,2);
  \draw (0,3) -- (1,4);
  \draw (1,3) -- (2,4);
  \draw (2,3) -- (3,4);
  \draw (3,3) -- (4,4);
  \draw (0,4) -- (1,3);
  \draw (1,4) -- (2,3);
  \draw (2,4) -- (3,3);
  \draw (3,4) -- (4,3);
  % 
  % Labeling
  % 
  \draw [] (0,0) node[left] {$(0,0)$};
  \draw [] (4,0) node[right] {$(1,0)$};
  \draw [] (0,4) node[left] {$(0,1)$};
  \draw [] (4,4) node[right] {$(1,1)$};
  %
  % C_n
  %
  \fill [pattern=dots, line width=0pt]
  (0,0) 
  -- 
  (1,0) 
  -- 
  (1,1) 
  -- 
  (0,1); 
  \fill [pattern=dots, line width=0pt]
  (2,0) 
  -- 
  (3,0) 
  -- 
  (3,1) 
  -- 
  (2,1); 
  \fill [pattern=dots, line width=0pt]
  (0,2) 
  -- 
  (1,2) 
  -- 
  (1,3) 
  -- 
  (0,3); 
  \fill [pattern=dots, line width=0pt]
  (2,2) 
  -- 
  (3,2) 
  -- 
  (3,3) 
  -- 
  (2,3); 
  %
  % Three types of triangle
  %
  % (a)
  \draw [] (0.8,2.5) node[inner sep=1.6pt, fill=white] {$a$};
  %
  % (b)
  \draw [] (1.20,2.5) node[inner sep=1pt, fill=white] {$b$};
  %
  % (c)
  \draw [] (1.5,2.20) node[inner sep=1pt, fill=white] {$c$};
\end{tikzpicture}
\caption{The triangulation of $U$ for $n = 2$. The dotted region
represents $\overline C_n \times \overline C_n$  where
$C_n = \left\{ \langle 0,0 \rangle, \langle 1,0 \rangle \right\}$.
The triangles labelled $a$, $b$, and $c$, in this order, are examples
of the three types of triangles distinguished in step \ref{Step4}.
}
\label{fig:ChessBoard}
\end{figure}

    \item Define $\overline f_n$ on each vertex $u$ of this triangulation by
      \[
        \overline f_n(u)
        \defeql
        \begin{cases}
          f_n(u) & \text{if $u \in \overline C_n \times \overline C_n$},\\
          u      & \text{otherwise.}
        \end{cases}
      \]

    \item\label{Step4} Define $\overline f_n$ on each triangle $t$ of the
      triangulation using the value of $\overline f_n$ on the vertices
      of $t$, depending on the following three cases:
      \begin{enumerate}
        \item\label{Case1} $t$ is contained in $\overline C_n \times \overline C_n$;
        \item\label{Case2} only the hypotenuse of $t$ is in $\overline C_n \times \overline C_n$;
        \item\label{Case3} otherwise.
      \end{enumerate}
  \end{enumerate}
  In all cases \eqref{Case1}--\eqref{Case3}, it suffices to define
  $\overline f_n$ on rational points
  $z \in t \cap \left( \Rat \times \Rat\right)$,
  since they are dense in $U$.

  For the case \eqref{Case1}, define $\overline f_n$ on $t$ by
  \[
    \overline f_n(z) \defeql f_n(z).
  \]
  
  For the cases \eqref{Case2} and \eqref{Case3}, let $u, v,w$ be the
  vertices of $t$. Then each rational point $z \in t$ can be
  expressed uniquely as
  \[
    \left\{
      \begin{aligned}
        &z = a_z u+b_z v+c_z w,\\
        &a_z+b_z+c_z =1,\\
        &a_z \ge 0,\; b_z\ge 0,\; c_z\ge 0,
      \end{aligned}
    \right.
  \]
  where $a_z,b_z,c_z \in \Rat$.
  Then, for the case \eqref{Case2},
  let us assume that $vw$ forms the hypotenuse of $t$, which is in
  $\overline C_n\times\overline C_n$ by the assumption.
  For each rational point $z\in t$, define $\overline f_n(z)$ by
  \[
    \overline f_n(z)
     \defeql
     \begin{cases}
       \overline f_n(u), & \text{if $a_z=1$}, \\[2ex]
       a_z \overline f_n(u)
        + (b_z+c_z) f_n\!\left( \frac{b_z v + c_z w}{b_z + c_z} \right),
                         & \text{otherwise}.
     \end{cases}
   \]
   Lastly, for the case \eqref{Case3},
   define $\overline f_n(z)$ by
   \[
     \overline f_n(z)
     \defeql
     a_z\overline f_n(u) + b_z\overline f_n(v) + c_z\overline f_n(w).
   \]
   Since the functions $\overline f_n$ defined on adjacent triangles
   coincide on their common boundary, we can glue them together to
   form a uniformly continuous function $\overline f_n$ on the entire
   $U$. Then it is clear that $\overline f_n$ extends $f_n$ and fixes
   $\partial U$.

   Now, define a uniformly continuous function $f \colon U \to U$ by 
   \begin{equation}
     \label{def:f}
     f \defeql \lim_{n\to\infty} \frac{1}{n+1}\left( \sum_{i=0}^{n}\overline f_i \right).
   \end{equation}
   Note that $f(\overline C_n \times \overline C_n) \subseteq \partial
   U$ for all $n \in \Nat$, and that $f$ fixes $\partial U$.
   Finally, define $h \colon U \to U$ by
   \begin{equation}
     \label{def:h}
     h \defeql R \circ f,
   \end{equation}
   where $R \colon U \to U$ denotes the $90^\circ$ rotation about the center of $U$.
   Since $f(\overline C_n \times \overline C_n) \subseteq \partial
   U$ for all $n \in \Nat$ and $R$ has no fixed point on $\partial U$,
   we have
   \begin{equation}
     \label{eq:NonFixed}
     \forall z \in \bigcup_{n \in \Nat} \overline C_n \times \overline C_n
     \left(  h(z) \neq z\right).
   \end{equation}
   By Brouwer's fixed-point theorem, $h$ has a fixed-point $(x,y) \in U$.
   By the binary expansion (cf.\ Remark \ref{rm:IVTBE}), there exist
   $\alpha,\beta \in \BTree$ such that
   $
   x = \sum_{i=0}^{\infty}2^{-(i+1)}\alpha(i),
   $
   and
   $
   y = \sum_{i=0}^{\infty}2^{-(i+1)}\beta(i).
   $
   If there exists $n \in \Nat$ such that $\overline{\alpha}n \notin T$
   and $\overline{\beta}n \notin T$, then
   $(x,y) \in \overline C_n \times \overline C_n$ so that
   $h(x,y) \neq (x,y)$ by \eqref{eq:NonFixed}, a contradiction.
   Thus, by $\LLPO$ (cf.\ Remark \ref{rm:IVTBE}),
   either $\forall n\in \Nat \left( \overline{\alpha}n \in T \right)$
   or $\forall n\in \Nat \left( \overline{\beta}n \in T \right)$.
   In either case, $T$ has an infinite path.
\end{proof}

Combining Theorem \ref{thm:WKLImpBrouwer} and Theorem \ref{thm:BrouwerImpWKL},
we obtain the following equivalence in $\ELZero$.
\begin{theorem}
  \label{thm:EquivBFTandWKL}
  The following are equivalent.
  \begin{enumerate}
    \item $\WKL$,
    \item Brouwer's fixed-point theorem.
  \end{enumerate}
\end{theorem}

\begin{remark}
  \label{rm:Replacement}
  The construction of the uniformly continuous function \eqref{def:h}
  can be used in place of the construction of a continuous function on
  $U$ without fixed points originally due to Orevkov
  \cite{Orevkov1963}, which is typically used to establish the
  equivalence between the following statements:
  \begin{enumerate}
    \item $\WKL$ and Brouwer's 
      fixed-point theorem \cite[Theorem 5.2]{ShiojiTanaka}
      \cite[IV.7.7]{SimpsonBook};
    \item the Weak Fan Theorem and Brouwer's Approximate
      Fixed-Point Theorem \cite[Theorem 7]{Veldman2009}; 
    \item Kleene's Alternative to the Weak Fan Theorem and
      the existence of a continuous function on $U$
      without fixed points \cite[Theorem 10]{Veldman2009}.
  \end{enumerate}
  Moreover, assuming the binary expansion of real numbers in $\UInt$,
  % 
  % \footnote{This is the case in the proof of $\WKL^{c}$ from $\IVT$ in
  % \cite[Theorem 3]{IshiharaEtalBinaryExpansion}.}
  % 
  the one-dimensional version of the proof of 
  Theorem~\ref{thm:BrouwerImpWKL} can be used to simplify the existing
  proof of the equivalence between the convex version of $\WKL$
  and the intermediate value theorem \cite[Theorem 3]{IshiharaEtalBinaryExpansion}.
\end{remark}

\subsection*{Acknowledgements}
The work is supported by JSPS KAKENHI Grant Number JP23K03197
and JP21KK0045.
\bibliographystyle{abbrv}
\bibliography{$HOME/refs.bib}

\newcommand{\noop}[1]{}
\begin{thebibliography}{10}

\bibitem{IshiharaEtalBinaryExpansion}
J.~Berger, H.~Ishihara, T.~Kihara, and T.~Nemoto.
\newblock The binary expansion and the intermediate value theorem in
  constructive reverse mathematics.
\newblock {\em Arch.\ Math.\ Logic}, 58:203--217, 2019.

\bibitem{Bishop-67}
E.~Bishop.
\newblock {\em Foundations of Constructive Analysis}.
\newblock McGraw-Hill, New York, 1967.

\bibitem{Bishop-Bridges-85}
E.~Bishop and D.~Bridges.
\newblock {\em Constructive Analysis}.
\newblock Springer, Berlin, 1985.

\bibitem{CourseInTopologicalCombinatorics}
M.~de~Longueville.
\newblock {\em A Course in Topological Combinatorics}.
\newblock Universitext. Springer, 2013.

\bibitem{Hendtlass2012FixedPT}
M.~Hendtlass.
\newblock Fixed point theorems in constructive mathematics.
\newblock {\em J. Log. Anal.}, 4:1--20, 2012.

\bibitem{IshiharaLLPOWKLHahn}
H.~Ishihara.
\newblock An omniscience principle, the {K\"onig} lemma and the {Hahn--Banach}
  theorem.
\newblock {\em Z. Math. Logik Grundlag. Math.}, 36:237--240, 1990.

\bibitem{ConstRevMatheCompactness}
H.~Ishihara.
\newblock Constructive reverse mathematics: compactness properties.
\newblock In L.~Crosilla and P.~Schuster, editors, {\em From Sets and Types to
  Topology and Analysis: Towards Practicable Foundations for Constructive
  Mathematics}, number~48 in Oxford Logic Guides, pages 245--267. Oxford
  University Press, 2005.

\bibitem{IshiharaWKLimpliesFan}
H.~Ishihara.
\newblock Weak {K}\"onig's lemma implies {B}rouwer's {F}an {T}heorem: {A Direct
  Proof}.
\newblock {\em Notre Dame J. Formal Log.}, 47:249--252, 2006.

\bibitem{Ishihara2009Relativize}
H.~Ishihara.
\newblock Relativization of {Real Numbers to a Universe}.
\newblock In S.~Lindstr{\"o}m, E.~Palmgren, K.~Segerberg, and
  V.~Stoltenberg-Hansen, editors, {\em Logicism, Intuitionism, and Formalism:
  What has Become of Them?}, pages 189--207. Springer, Dordrecht, 2009.

\bibitem{Kawai2026WKLMIN}
T.~Kawai.
\newblock The strength of the minimum principle in constructive reverse
  mathematics.
\newblock \url{http://arxiv.org/abs/2610.12964}, 2026.

\bibitem{KleeneVesley}
S.~C. Kleene and R.~E. Vesley.
\newblock {\em The foundations of intuitionistic mathematics, especially in
  relation to recursive functions}.
\newblock North-Holland, Amsterdam, 1965.

\bibitem{Orevkov1963}
V.~Orevkov.
\newblock A constructive mapping from the square onto itself displacing every
  constructive point.
\newblock {\em Soviet Math. Doklady}, 4:1253--1256, 1963.

\bibitem{ShiojiTanaka}
N.~Shioji and K.~Tanaka.
\newblock Fixed point theorem in weak second-order arithmetic.
\newblock {\em Ann.\ Pure Appl.\ Logic}, 47:167--188, 1990.

\bibitem{SimpsonBook}
S.~G. Simpson.
\newblock {\em Subsystems of second order arithmetic}.
\newblock Perspectives in Logic. Cambridge University Press, Cambridge, second
  edition, 2009.

\bibitem{SpernerLemma}
E.~Sperner.
\newblock Neuer beweis f\"{u}r die invarianz der dimensionszahl und des
  gebietes.
\newblock {\em Abh.\,Math.\,Semin.\,Univ.\,Hambg}, 6:265--272, 1928.

\bibitem{Todd1976}
M.~J. Todd.
\newblock {\em The Computation of Fixed Points and Applications}.
\newblock Springer, Berlin, Heidelberg, 1976.

\bibitem{ConstMathI}
A.~S. Troelstra and D.~{van Dalen}.
\newblock {\em Constructivism in Mathematics: An Introduction. Volume {I}},
  volume 121 of {\em Studies in Logic and the Foundations of Mathematics}.
\newblock North-Holland, Amsterdam, 1988.

\bibitem{Veldman2009}
W.~Veldman.
\newblock Brouwer's {A}pproximate {F}ixed-{P}oint {T}heorem is {E}quivalent to
  {B}rouwer's {F}an {T}heorem.
\newblock In S.~Lindstr{\"o}m, E.~Palmgren, K.~Segerberg, and
  V.~Stoltenberg-Hansen, editors, {\em Logicism, Intuitionism, and Formalism:
  What has Become of Them?}, pages 277--299. Springer, Dordrecht, 2009.

\bibitem{VeldmanFANKleene}
W.~Veldman.
\newblock Brouwer's {F}an {T}heorem as an axiom and as a contrast to {K}leene's
  alternative.
\newblock {\em Arch.\ Math.\ Logic}, 53(5):621--693, 2014.

\bibitem{VickHomology}
J.~W. Vick.
\newblock {\em Homology theory : an introduction to algebraic topology}.
\newblock Number 145 in Graduate texts in mathematics. Springer-Verlag, second
  edition, 1994.

\end{thebibliography}
\end{document}